\documentclass[11pt]{amsart}

\newtheorem{theorem}{Theorem}[section]
\newtheorem{lemma}[theorem]{Lemma}
\newtheorem{proposition}[theorem]{Proposition}

\theoremstyle{definition}
\newtheorem{definition}[theorem]{Definition}
\newtheorem{example}[theorem]{Example}

\theoremstyle{remark}
\newtheorem{remark}[theorem]{Remark}

\numberwithin{equation}{section}

\begin{document}

\title[Almost limited sets in Banach lattices]
 {Almost limited sets in Banach lattices}
\author[J.X. Chen]
{Jin Xi Chen }

\address{College of Mathematics and Information Science, Shaanxi Normal
University, Xi'an 710062, P.R. China}
\address{Department of Mathematics, Southwest Jiaotong
University, Chengdu 610031, P.R. China}
 \email{jinxichen@home.swjtu.edu.cn}

\author[Z.L. Chen]
{Zi Li Chen}
\address{Department of Mathematics, Southwest Jiaotong
University, Chengdu 610031, P.R. China}
\email{zlchen@home.swjtu.edu.cn}

\author[G.X. Ji]
{Guo Xing Ji}
\address{College of Mathematics and Information Science, Shaanxi Normal
University, Xi'an 710062, P.R. China}
\email{gxji@snnu.edu.cn}

\thanks{The first author was supported in part by NSFC (No.11301285) and the Fundamental Research Funds for the Central Universities (SWJTU11CX154). The second author was supported in part by the Fundamental Research Funds for the Central Universities (SWJTU12ZT13). The third author was support in part by NSFC (No.11371233).}

\subjclass[2000]{Primary 46B42; Secondary 46B50, 47B65}

\keywords{almost limited set, the wDP$^{*}$ property, almost Dunford-Pettis operator, positive Schur property, Banach lattice}

\begin{abstract}
We introduce and study the class of almost limited sets in Banach lattices, that is, sets on which every disjoint weak$^{*}$ null sequence of functionals converges uniformly to zero. It is established that a Banach lattice has order continuous norm if and only if almost limited sets and $L\,$-weakly compact sets coincide. In particular, in terms of almost Dunford-Pettis operators into $c_{0}$, we give an operator characterization of those $\sigma$- Dedekind complete Banach lattices whose relatively weakly compact sets are almost limited, that is, for a $\sigma$-Dedekind Banach lattice $E$,  every relatively weakly compact set in $E$ is  almost limited  if and only if every continuous linear operator $T:E\rightarrow c_{\,0}$ is an almost Dunford-Pettis operator.
\end{abstract}

\maketitle \baselineskip 4.95mm

\section{Introduction}
\par Throughout this paper $X,\,Y$ will denote real Banach spaces, and $E,\,F$ will denote real Banach lattices. $B_{X}:=$ the closed unit ball of $X$. $sol(A)$ denotes the solid hull of a subset $A$ of a Banach lattice. The positive cone of $E$ will be denoted by $E^{\,+}$.
\par Let us recall that a bounded subset $A$ of $X$ is called a \textit{Dunford-Pettis set} (resp. a \textit{limited set}) in $X$ if each weakly null sequence in $X^{*}$ (resp. weak$^{*}$ null sequence in $X^{*}$) converges uniformly to zero  on $A$. Clearly, every limited set in $X$ is a Dunford-Pettis set, but the converse is not true in general.  We say that $X$ has the \textit{Dunford-Pettis property }whenever $x_{n}\xrightarrow {w} 0$ in $X$ and $f_{n}\xrightarrow {w} 0$ in $X^{*}$ imply $\lim_{\,n} f_{n}(x_{n})=0$, equivalently, every relatively weakly compact set in $X$ is a Dunford-Pettis set, alternatively, every weakly compact operator $T:X\rightarrow c_{0}$ is a Dunford-Pettis operator. If all limited sets in $X$ are relatively compact, then $X$ is said to be a \textit{Gelfand-Phillips space}. It is well-known that all separable Banach spaces and all weakly compactly generated spaces are Gelfand-Phillips spaces. Note that a $\sigma$-Dedekind complete Banach lattice $E$ is a Gelfand-Phillips space if and only if the norm of $E$ is order continuous (cf. \cite{B}).  $X$ has the \textit{Dunford-Pettis$^{*}$ property} (the DP$^{*}$ property for short\,) whenever every relatively weakly compact set in $X$ is limited, in other words, for any weakly null sequence $(x_{n})$ in $X$ and any weak$^{*}$ null sequence $(f_{n})$ in $X^{*}$, $\lim_{\,n} f_{n}(x_{n})=0$.  The DP$^{*}$ property, introduced first by Borwein, Fabian and Vanderwerff \cite{BFV}, is stronger than the Dunford-Pettis property. Carri\'{o}n, Galindo and Louren\c{c}o \cite{CGL} showed that $X$ has the DP$^{*}$ property if, and only if, every continuous linear operator $T:X\rightarrow c_{0}$ is a Dunford-Pettis operator.

\par Recall that a Banach lattice $E$ has the \textit{positive Schur property} (i.e., \textit{weak Schur property}) if every weakly null sequence with positive terms is norm null, equivalently, every disjoint weakly null sequence in $E$ is norm null. A continuous operator $T$ from $E$ into a Banach space is called \textit{almost Dunford-Pettis} (\cite{Sa}) if $\|T(x_{n})\|\rightarrow 0$ for every disjoint, weakly null sequence $(x_{n})$ in $E$. We say that $E$ has the \textit{weak Dunford-Pettis property} (wDP property for short) if every weakly compact operator from $E$ into any Banach space $Y$ is almost Dunford-Pettis (cf. \cite{Le}). It is obvious that the Dunford-Pettis property or the positive Schur property imply the weak Dunford-Pettis property. As Wnuk pointed out in \cite{W4}, $E$ has the weak Dunford-Pettis property if and only if every weakly compact operator from $E$ into $c_{0}$ is almost Dunford-Pettis, equivalently, for every disjoint weakly null sequence $(x_{n})$ in $X$ and every weakly null sequence $(f_{n})$ in $X^{*}$, $\lim_{\,n} f_{n}(x_{n})=0$.

\par Recently, Bouras \cite{Kh} considered the disjoint version of Dunford-Pettis sets and introduced the class of almost Dunford-Pettis sets in Banach lattices.      Following Bouras, a bounded subset $A$ of a Banach lattice $E$ is said to be an \textit{almost Dunford-Pettis set} if every disjoint weakly null sequence $(f_{n})$ of $ E^{\,*}$ converges uniformly to zero on $A$. He showed that a Banach lattice $E$ has the weak Dunford-Pettis property if and only if every relatively weakly compact set in $E$ is almost Dunford-Pettis (\cite{Kh}).

\par Inspired by Carri\'{o}n, Galindo and Louren\c{c}o \cite{CGL}, we may ask under what conditions every continuous operator from a Banach lattice $E$ into $c_0$  is almost Dunford-Pettis.  In this paper, using disjoint sequence techniques we consider the disjoint version of limited sets, i.e., the almost limited sets in Banach lattices (\,Definition \ref {Definition 1}).  We introduce the weak Dunford-Pettis$^{*}$ property (wDP$^{*}$ property for short) which is shared by those Banach lattices whose relatively weakly compact subsets are almost limited. In terms of almost Dunford-Pettis operators into $c_{0}$, we also give an operator characterization of Banach lattices with the wDP$^{*}$ property, that is, a $\sigma$-Dedekind Banach lattice $E$ has the wDP$^{\,*}$ property if and only if every continuous operator $T:E\rightarrow c_{\,0}$ is an almost Dunford-Pettis operator. (Theorem \ref{Theorem 2.2}).

\par Our notions are standard. The reader should see \cite{W1, W2, W4} for the (positive) Schur property and the (weak) Dunford-Pettis property of Banach lattices. For the theory of Banach lattices and operators, we refer  the reader to the monographs \cite{AB, M}.

\section{Almost Limited Sets in Banach Lattices}

\par It should be noted that in a Banach lattice (or in its dual) the lattice operations fail to be weakly (\,resp. weak$^{*}$) sequentially continuous  in general. Let us recall that every disjoint sequence in the solid hull of a relatively weakly compact subset of a Banach lattice $E$ converges weakly to zero (\cite[Theorem 13.3]{AB}). Therefore, if $(x_{n})$ is a disjoint, weakly convergent sequence in $E$, then naturally the sequences $(x_{n})$, $(|\,x_{n}|)$, $(x_{n}^{+})$, $(x_{n}^{-})$ all converge weakly to zero. However, as we shall see from the following example, $w^{*}$-convergent disjoint sequences in the dual can not be that congenial.

\begin{example}\label{Example 1}
\begin{enumerate}
  \item Let $(\delta_{\frac{1}{n}})$ be a sequence of evaluation functionals on $C\,[0,\,1]$. Clearly, $(\delta_{\frac{1}{n}})$ is a disjoint sequence and $\delta_{\frac{1}{n}}\xrightarrow {w^{*}}\delta_{0}\neq 0$ in $C\,[0,\,1]^{\,*}$.
  \item Let $f_{n}\in c^{\,*}=\ell^{1}$ $(n=1,2,3,\cdot\cdot\cdot)$ be defined as follows:
  $f_{1}=(0,1,-1,0,\cdot\cdot\cdot)$, $f_{2}=(0,0,0,1,-1,0\cdot\cdot\cdot)$,\quad $\cdot\cdot\cdot,\quad f_{n}=(0,\cdot\cdot\cdot,0,1_{(2n)},-1_{(2n+1)},0,\cdot\cdot\cdot).$
  Then $(f_{n})$ is a disjoint, weak$^{*}$ null sequence in $c^{\,*}$, but $(|\,f_{n}|)$ does not weak$^{*}$ converge to zero. Indeed, $|\,f_{n}|(\textbf{1})=\sup_{x\in [\textbf{-1},\,\textbf{1}]}|\,f_n(x)|=\sup_{x\in B_c}|\,f_n(x)|=\|\,f_n\|=2$, where $\textbf{1}:=(1,1,1,\cdot\cdot\cdot)\in c$.
\end{enumerate}
\end{example}
For a $\sigma$-Dedekind complete Banach lattice, the situation is quite different. More precisely, we have the following lemma.

\begin{lemma}\label{Lemma 1}
Let $E$ be a $\sigma$-Dedekind complete Banach lattice, and let $(f_{n})$ be a $w^{\,*}$-convergent sequence of $E^{\,*}$. If $(g_n)$ is  a disjoint sequence of $E^{\,*}$ satisfying $|\,g_n|\leq |\,f_n|$ for each $n\in \mathbb{N}$, then the sequences $(g_{n})$, $(|\,g_{n}|)$, $(g_{n}^{\,+})$, $(g_{n}^{\,-})$ all weak$^{\,*}$ converge to zero. In particular, if $(f_n)$ is a disjoint $w^*$-convergent sequence in its own right, then the sequences $(f_{n})$, $(|\,f_{n}|)$, $(f_{n}^{\,+})$, $(f_{n}^{\,-})$ are all weak$^*$ null.
\end{lemma}

\begin{proof}
 Let $x\in E^{\,+}$, and let $\varepsilon>0$. Since $E$ is $\sigma$-Dedekind complete and $(f_{n})$ is a $w^{\,*}$-convergent sequence of $E^{\,*}$, there exists $0\leq f \in E^{\,*}$ lying in the ideal generated by $(f_{n})$ in $E$ such that $$(|\,f_n|-f)^{+}(x)<\varepsilon$$holds for all $n\in\mathbb{N}$ (\,\cite{Burk}; cf. \cite[Theorem 13.11]{AB}). Therefore, we have
 \begin{eqnarray*}
|\,g_{n}(x)|\leq |\,g_{n}|(x)&=&(|\,g_n|-f)^{+}(x)+(|\,g_n|\wedge f)(x)\\&\leq & (|\,f_n|-f)^{+}(x)+(|\,g_n|\wedge f)(x)\\&<&\varepsilon+(|\,g_n|\wedge f)(x).
 \end{eqnarray*}Because $(|\,g_n|\wedge f)$ is an order bounded disjoint sequence, we have $|\,g_n|\wedge f\xrightarrow {w}0$, and hence $\limsup|\,g_{n}|(x)\leq\varepsilon$. This implies $g_n\xrightarrow {w^{\,*}}0$ and $|\,g_n|\xrightarrow {w^{\,*}}0$. Finally, the inequalities $g_{n}^{\,+}\leq|\,g_n|$ and $g_{n}^{\,-}\leq|\,g_n|$ finish the proof.
\end{proof}
Next we give the definition of an almost limited set in a Banach lattice, which is the disjoint version of the limited set, and is in a sense also the $w^*$\,- counterpart of the almost Dunford-Pettis set.

\begin{definition}\label{Definition 1}
A norm bounded subset $A$ of $E$ is said to be an \textit{almost limited set} if every disjoint, weak$^{*}$ null sequence $(f_{n})$ of $ E^{\,*}$ converges uniformly to zero on $A$, that is, $\sup_{x\in A}|\,f_{n}(x)|\rightarrow0$.
\end{definition}
Now we are in a position to give some examples of almost limited sets and distinguish the class of almost limited sets from the classes of relatively (weakly) compact sets, limited sets and (almost) Dunford-Pettis sets, etc.
\begin{remark}\label{Remark 1}

  (1) By the definition of an almost limited set, every order interval in a Banach lattice is almost limited if, and only if, $|\,f_n|\xrightarrow {w^{\,*}}0$ for each disjoint $w^*$-null sequence in $E^{\,*}$ (\,\cite[Theorem 11.11]{AB}). Then, by Lemma \ref{Lemma 1}, in a $\sigma$-Dedekind complete Banach lattice every order interval is an almost limited set. If $E$ is not $\sigma$-Dedekind complete, an order interval of $E$ is not necessarily almost limited. We can see this from Example \ref{Example 1}(2).
  \par (2) It is obvious that all relatively compact sets and all limited sets in a Banach lattice are almost limited. The converse does not hold in general. For example, by $(1)$ $B_{\ell^{\,\infty}}$ is almost limited, but $B_{\ell^{\,\infty}}$ is not either compact or limited.  We can also find a counterexample in a Banach lattice with order continuous norm. For instance, $L_1[0,\,1]$ is a Gelfand-Phillips space, but there exists an order interval in $L_1[0,\,1]$ which is not compact, since $L_1[0,\,1]$ is not a discrete space (cf. \cite[Corollary 21.13]{AB1}). On the other hand, since $L_1[0,\,1]$ has order continuous norm (and hence Dedekind complete), by $(1)$  each order interval of $L_1[0,\,1]$ is almost limited.
  \par(3) Clearly, every almost limited set is an almost Dunford-Pettis set, but the converse is not true in general. For instance, $B_{c_{0}}$ is a Dunford-Pettis set (and hence an  almost Dunford-Pettis set), but $B_{c_{0}}$ is not almost limited. It should be noted that in a Grothendieck Banach lattice the class of almost limited sets and the class of almost Dunford-Pettis sets are the same.
  \par(4) A relatively weakly compact set need not be almost limited, and vice versa. For instance, $B_{\ell^{\,2}}$ is weakly compact, but not almost limited. On the other hand, $B_{\ell^{\,\infty}}$ is not weakly compact, but by $(1)$ $B_{\ell^{\,\infty}}$ is almost limited.
  \par(5) It is well known that every limited set is conditionally weakly compact \cite{BD}, and the Josefson-Nissenzweig theorem precludes any possibility of the closed unit ball of an infinite dimensional Banach space being limited. However, $B_{\ell^{\,\infty}}$ is indeed almost limited, and by Rosenthal's $\ell^{\,1}$ theorem $B_{\ell^{\,\infty}}$ is not conditionally weakly compact.

\end{remark}

\par Let $F$ be a Banach sublattice of a Banach lattice $E$. It may happen that a subset $A$ of $F$ is almost limited in $E$, but fails to be almost limited in $F$. For example, Phillips' lemma shows that $B_{c_{0}}$ is a limited set in $\ell^{\,\infty}$, but $B_{c_{0}}$ is not  almost limited in $c_{\,0}$. It should also be noted that the solid hull of an almost limited set in a Banach lattice  is not necessarily almost limited. For instance, the singleton $\{\textbf{1}\}$ is certainly almost limited in $c$, but  $sol\{\textbf{1}\}$$\,=\,$$B_c=[-\textbf{1},\,\textbf{1}]$ is not almost limited (see Example \ref{Example 1} (2)). A further investigation will be made in Remark \ref{Remark 2} (2). The following theorem characterizes solid sets being almost limited.

\begin{theorem}\label{Theorem 1}
Let $S$ be a norm bounded solid subset of a $\sigma$-Dedekind complete Banach lattice $E$. Then the following statements are equivalent.
\par \begin{enumerate}
  \item $S$ is an almost limited set in $E$.
  \item For each disjoint sequence $(x_n)$ in $S$ and each disjoint $w^*$-null sequence $f_n$ in $E^{\,*}$, we have $\lim f_{n}(x_n)=0$.
  \item For each disjoint sequence $(x_n)$ in $S\cap E^{\,+}$ and each disjoint $w^*$-null sequence $f_n$ in $(E^{\,*})^+$, we have $\lim f_{n}(x_n)=0$.
\end{enumerate}
\end{theorem}

\begin{proof}
$(1)\Rightarrow(2)\Rightarrow(3)$ Obvious.
\par $(3)\Rightarrow(1)$ Let $(f_n)$ be an arbitrary disjoint $w^{\,*}$-null sequence in $E^{\,*}$. To finish the proof, we have to show that $\sup_{x\in S}|\,f_{n}(x)|\rightarrow0$. Assume by way of contradiction that $\sup_{x\in S}|\,f_{n}(x)|$ does not converge to 0 as $n\rightarrow\infty$. Then, by passing to a subsequence if necessary, we can suppose that there would exist some $\varepsilon>0$ such that $\sup_{x\in S}|\,f_{n}(x)|>\varepsilon$ for all $n\in\mathbb{ N}$. Note that the equality $\sup_{x\in S}|\,f_{n}(x)|=\sup_{0\leq x\in S}|f_{n}|(x)$ holds, since $S$ is solid. Therefore, for each $n$ choose some $0\leq x_n\in S$ satisfying $|\,f_n|(x_n)>\varepsilon$. In view of Lemma \ref{Lemma 1}, we have $|\,f_n|\xrightarrow{w^*}0$. Let $n_{1}=1$. Because $|\,f_{n}|(4x_{n_{1}})\rightarrow0$ $(n\rightarrow\infty)$, there exists some $1<n_{2}\in\mathbb{N}$ such that $|\,f_{n_{2}}|(4x_{n_{1}})<2^{-1}$. Again, since $|\,f_n|(4^2\sum_{k=1}^{\,2}x_{n_{k}})\rightarrow0$ $(n\rightarrow\infty)$, choose some $n_{3}\in\mathbb{N}$ $(n_1<n_2<n_3)$ satisfying $|\,f_{n_3}|(4^2\sum_{k=1}^{\,2}x_{n_{k}})<2^{-2}$. It is easy to see that, by induction, we can find a strictly increasing subsequence $(n_{k})_{k=1}^{\infty}$ of $\mathbb{N}$ such that $|\,f_{n_{m+1}}|(4^{m}\sum_{k=1}^{\,m}\,x_{n_{k}}\,)<2^{-m}$ for all $m\in\mathbb{N}$. Let $$x=\sum_{k=1}^{\,\infty}2^{-k}x_{n_{k}},\quad y_{m}=(x_{n_{m+1}}-4^{m}\sum_{k=1}^{\,m}\,x_{n_{k}}-2^{-m}x)^{+}.$$Then, in view of \cite[Lemma 13.4]{AB} $(y_{m})$ is a disjoint sequence, and $(y_{m})\subset S\cap E^{\,+}$ because $0\leq y_{m}\leq x_{n_{m+1}}\in S$ and $S$ is solid. Now, we have
\begin{eqnarray*}
    |f_{n_{m+1}}|(y_m)&=&|f_{n_{m+1}}|\left(x_{n_{m+1}}-4^{m}\sum_{k=1}^{\,m}\,x_{n_{k}}-2^{-m}x\right)^{+}\\&\geq& |f_{n_{m+1}}|\left(x_{n_{m+1}}-4^{m}\sum_{k=1}^{\,m}\,x_{n_{k}}-2^{-m}x\right)
    \\&=&|f_{n_{m+1}}|(x_{n_{m+1}})-|f_{n_{m+1}}|\left(4^{m}\sum_{k=1}^{\,m}\,x_{n_{k}}\right)-2^{-m}|f_{n_{m+1}}|(x)\\&>&\varepsilon-2^{-m}-2^{-m}|f_{n_{m+1}}|(x).
\end{eqnarray*}
Note that $|\,f_{n_{m}}|\xrightarrow{w^*}0$ $(m\rightarrow\infty)$. Hence, $\liminf|f_{n_{m+1}}|(y_m)\geq\varepsilon>0$. On the other hand, since $(y_m)$ is a disjoint sequence of $S\cap E^{\,+}$ and $|f_{n_{m}}|$ is a disjoint $w^*$-null sequence in $(E^{\,*})^{+}$, by hypothesis we have $\lim_{m}|f_{n_{m+1}}|(y_m)=0$. This leads to a contradiction, and the proof is completed.
\end{proof}

\par Let us recall that a norm bounded subset $A$ of a Banach lattice $E$ is called to be \textit{$L\,$-weakly compact} if $\|x_n\|\rightarrow0$ for every disjoint sequence $(x_{n})$ contained in the solid hull of $A$ (cf. \cite[Definition 3.6.1]{M}). Every $L\,$-weakly compact set is relatively weakly compact set, but the converse does not hold in general. In an $L$-space, $L\,$-weakly compact sets and relatively weakly compact sets coincide. More generally, every relatively weakly compact subset of $E$ is $L$-weakly compact if, and only if, $E$ has the positive Schur property (\cite[Corollary 3.6.8]{M}).  As we see from Remark \ref{Remark 1} (4), an almost limited set need not be relatively weakly compact (hence not $L\,$-weakly compact) even if the Banach lattice is Dedekind complete. The following theorem deals with the relationship of $L\,$-weakly compact sets with almost limited sets.
\begin{theorem}\label{Theorem 2}
(1) Every $L$-weakly compact set in a Banach lattice $E$ is an almost limited set.
\par (2) The norm of $E$ is order continuous if, and only if,  each almost limited set in $E$ is $L$-weakly compact.

\end{theorem}

\begin{proof}

   (1)  Let $A$ be an $L\,$-weakly compact subset of $E$, and let $(f_{n})$ be any disjoint $w^*$-null sequence of $E^{*}$. By Proposition 3.6.2 of \cite{M} we have $\sup_{x\in A}|\,f_n|(|\,x|)\rightarrow0$. The inequality $\sup_{x\in A}|\,f_n(x)|\leq\sup_{x\in A}|\,f_n|(|\,x|)$ implies that $(f_{n})$ converges uniformly to zero on $A$, that is, $A$ is an almost limited set.

   \par (2) Assume that $E$ has order continuous norm. Let $B$ be an almost limited set in $E$. To prove that $B$ is $L\,$-weakly compact, by Proposition 3.6.2 of \cite{M} we only need to show that $\rho_{B}(f_{n})\rightarrow0$ for every norm bounded disjoint sequence $(f_n)$ of $E^{\,*}$, where $\rho_{B}(f)$ is defined by
      $$\rho_{B}(f)=\sup\{|f|(|\,x|):x\in B\}=\sup\{|g(x)|:|g|\leq|f|,\,x\in B\}$$ for every $f\in E^{\,*}$. Assume by way of contradiction that $(\rho_{B}(f_{n}))$ does not tend to 0 as $n\rightarrow\infty$ for some norm bounded disjoint sequence $(f_n)$ of $E^{\,*}$. Then, by passing to a subsequence if necessary, we can suppose that there would exist some $\varepsilon>0$ satisfying $\rho_{B}(f_n)=\sup\{|\,f_n|(|\,x|):x\in B\}>\varepsilon$ for all $n$. For each $n$ choose some $x_{n}\in B$ and some $g_{n}\in E^{\,*}$ with $|\,g_n|\leq|\,f_n|$ such that $|\,g_{n}(x_{n})|>\varepsilon$. Clearly, $(g_n)$ is likewise a norm bounded disjoint sequence. It follows from the order continuity of the norm of $E$ that $g_n\xrightarrow{w^{*}}0$ (\cite[Corollary 2.4.3]{M}). Since $B$ is almost limited, $(g_n)$ converges uniformly to 0 on $B$, which implies that $|\,g_{n}(x_{n})|\rightarrow0$. This leads to a contradiction.
\par Now assume that every almost limited set in $E$ is $L\,$-weakly compact. To establish that the norm of $E$ is order continuous, it suffices to show that every disjoint sequence $(f_n)$ from $B_{E^{\,*}}$ is $w^*$-null (\cite[Corollary 2.4.3]{M}). To this end, let $x\in E$. Clearly, the singleton $\{x\}$ is almost limited, and hence by hypothesis $\{x\}$ is $L\,$-weakly compact. By Proposition 3.6.2 of \cite{M}, we have $\rho_{x}(f_{n})=|\,f_n|(|\,x|)\rightarrow0$. The inequality $|\,f_{n}(x)|\leq|\,f_n|(|\,x|)$ finishes the proof.
\end{proof}

\begin{remark}\label{Remark 2}
(1) It should be noted that, in a $\sigma$-Dedekind complete Banach lattice $E$, every limited set is relatively compact (i.e., $E$ is a Gelfand-Phillips space)  if, and only if, the norm of $E$ is order continuous (cf.\,\cite{B}).
\par (2) From the remarks just preceding Theorem \ref{Theorem 1} we see that the solid hull of an almost limited set is not necessarily almost limited.  If $E$ has order continuous norm, then by Theorem \ref{Theorem 2} (2) the solid hull of an almost limited set in $E$ is almost limited, since the solid hull of an $L\,$-weakly set is likewise $L\,$-weakly compact. However, the converse does not hold in general. For instance, every norm bounded set in $\ell^{\,\infty}$ is almost limited, but the norm of $\ell^{\,\infty}$ is not order continuous.
\end{remark}

\par Let us recall that a norm bounded subset $B$ of $X^{*}$ is called an $L$\,-set whenever every weakly null sequence $(x_n)$ of $X$ converges uniformly to zero on the set $B$, that is, $\sup_{f\in B}|f(x_n)|\rightarrow0$ (cf. \cite{Em}). Recently, Aqzzouz and Bouras \cite{Aq} introduced the class of almost $L$\,-sets in Banach lattices. A norm bounded subset $B$ of the dual $E^{*}$ of a Banach lattice $E$ is said to be an almost $L$\,-set if every disjoint, weakly null sequence $(x_n)$ of $E$ converges uniformly to zero on $B$. In $E^{*}$  the following implications are clear:$$\textrm{almost limited set}\Longrightarrow \textrm{almost Dunford-Pettis set}\Longrightarrow \textrm{almost}\,  L\,\textrm{-set}.$$From Corollary 2.12 of \cite {Kh} and Theorem \ref{Theorem 2} it follows that if $E^{*}$ has order continuous norm, then the class of almost limited sets and the class of almost Dunford-Pettis sets coincide in $E^*$. Indeed, we can say more.

\begin{theorem}\label{Theorem 2.8}
Let $E$ be a Banach lattice. If the norm of $E^*$ is order continuous, then in $E^*$ the class of almost limited sets,  the class of almost Dunford-Pettis sets and the class of almost $L$\,-sets are the same.
\end{theorem}

The proof of Theorem \ref{Theorem 2.8} is based on Theorem \ref{Theorem 2} and the following lemma.
\begin{lemma}\label{Lemma 2.9}
Every almost $L$\,-set in $E^*$ is $L$\,-weakly compact if and only if $E^*$ has order continuous norm.
\end{lemma}
\begin{proof}
 Assume that the norm of $E^*$ is order continuous. Let $L$ be an almost $L$\,-set in $E^*$. To establish  $L\,$-weak compactness of $L$, by Proposition 3.6.3 of \cite{M} we only have to show that for every norm bounded disjoint sequence $(x_n)$ of $E$,
      $$\rho_{L}(x_n)=\sup\{|f|(|x_{n}|):f\in L\}=\sup\{|f(y)|:|y|\leq|x_n|,\,f\in L\}\rightarrow0.$$ Assume by way of contradiction that $(\rho_{L}(x_{n}))$ does not converge to 0 for some norm bounded disjoint sequence $(x_n)$ of $E$. Then, by passing to a subsequence if necessary, we can suppose that there  exists some $\varepsilon>0$ such that $$\rho_{L}(x_n)=\sup\{|f(y)|:|y|\leq|x_n|,\,f\in L\}>\varepsilon$$for all $n$. For each $n$ choose some $f_{n}\in L$ and some $y_{n}\in E$ with $|y_n|\leq|x_n|$ satisfying $|f_{n}(y_{n})|>\varepsilon$. We can see that $(y_n)$ is  a norm bounded disjoint sequence. The order continuity of the norm of $E^*$ implies that $y_n\xrightarrow{w}0$ (\cite[Corollary 2.4.14]{M}). Since $L$ is an almost $L$\,-set in $E^*$, the disjoint weakly null sequence $(y_n)$ converges uniformly to 0 on $B$, which implies that $|f_{n}(y_{n})|\rightarrow0$. This contradicts with $|f_{n}(y_{n})|>\varepsilon$. Therefore, $L$ is $L$\,-weakly compact.
\par For the converse, assume that each almost $L$\,-set in $E^*$ is $L$\,-weakly compact. To prove that $E^*$ has order continuous norm, we need only to show that every disjoint sequence $(x_n)$ of $B_{E}$ is weakly null (\cite[Corollary 2.4.14]{M}). For this, let $f\in E^*$. Clearly, the singleton $\{f\}$ is an almost $L$\,-set in $E^*$, and  by hypothesis $\{f\}$ is $L\,$-weakly compact. In view of Proposition 3.6.3 of \cite{M}, we have $\rho_{f}(x_{n})=|f|(|x_n|)\rightarrow0$. Clearly, $|f(x_n)|\rightarrow0$, as desired.
\end{proof}

\section{The Weak Dunford-Pettis$^{\,*}$ Property of Banach Lattices}

Recall that a Banach space is said to have the DP$^*$ property if all relatively weakly compact sets are limited. Similarly, we introduce the so-called wDP$^*$ property of a Banach lattice.

\begin{definition}
A Banach lattice $E$ is called to have the \textit{weak Dunford-Pettis$^{\,*}$ property} (\textit{wDP$^{\,*}$ property} for short) if every relatively weakly compact set in $E$ is almost limited.
\end{definition}
\par In other words, $E$ has the wDP$^{*}$ property  if and only if for each weakly null sequence $(x_n)$ in $E$ and each disjoint $w^{*}$-null sequence in $E^{\,*}$, $f_{n}(x_{n})\rightarrow0$.
\par When the Banach lattice is $\sigma$-Dedekind complete, we can characterize the wDP$^*$ property in terms of disjoint sequences.
\begin{theorem}\label{Theorem 3}
For a $\sigma$-Dedekind complete Banach lattice $E$, the following statements are equivalent:
\begin{enumerate}
  \item $E$ has the  wDP$^{\,*}$ property.
  \item For each disjoint weakly null sequence $(x_n)\subset E$ and each disjoint $w^{*}$-null sequence $(f_n)\subset E^{\,*}$, we have $f_{n}(x_{n})\rightarrow0$.
  \item For each disjoint weakly null sequence $(x_n)\subset E^{+}$ and each disjoint $w^{*}$-null sequence $(f_n)\subset (E^{\,*})^+$, we have $f_{n}(x_{n})\rightarrow0$.
  \item The solid hull of every relatively weakly compact set in $E$ is almost limited.
  \end{enumerate}
\end{theorem}

\begin{proof}
Only $(3)\Rightarrow(4)$ needs a proof. To this end, let $W$ be a relatively weakly compact set in $E$. It should be noted that each disjoint sequence in the solid hull $sol(W)$ of $W$ converges weakly to 0 (see \cite[Theorem 13.3]{AB}). So, for every disjoint sequence $(x_n)$ in  $sol(W)\cap E^{\,+}$ and every disjoint $w^*$-null sequence $(f_n)$ of $(E^{\,*})^+$, by hypothesis we have $f_{n}(x_{n})\rightarrow0$. Hence, it follows from Theorem \ref{Theorem 1} that $sol(W)$ is almost limited.
\end{proof}

\par Since every Banach lattice  with order continuous norm is a Gelfand-Phillips space, $E$ has the Schur property if and only if $E$ has both order continuous norm and the DP$^*$ property. Let us recall that a Banach lattice $E$ has the positive Schur property if and only if every relatively weakly compact subset of $E$ is $L$-weakly compact (\cite[Corollary 3.6.8]{M}). Every Banach lattice with the positive Schur property is a KB-space (and hence has order continuous norm). Therefore, by Theorem \ref{Theorem 2} we have the following easy result and omit the proof.
\begin{proposition}\label{Proposition 1}
A Banach lattice $E$ has the positive Schur property if, and only if, $E$ has order continuous norm and the wDP$^{\,*}$ property.
\end{proposition}

\begin{remark}
(1) It is obvious that the (positive) Schur property and the DP$^*$ property imply the wDP$^*$ property.  However, $\ell^{\infty}\oplus L^{1}[0,\,1]$ is a Dedekind complete  Banach lattice with the wDP$^{*}$ property which has neither the (positive) Schur property nor the DP$^*$ property since $\ell^{\infty}$ has the DP$^*$ property without the positive Schur property, whereas $L^{1}[0,\,1]$ has the positive Schur property without the DP$^*$ property (a separable Banach space with the DP$^*$ property must have the Schur property).
\par(2) Clearly, every Banach lattice with the wDP$^*$ property has the wDP property, but the converse is not necessarily true. For instance, $c_{\,0}$ enjoys the Dunford-Pettis property (hence the wDP property), but by Proposition \ref{Proposition 1} $c_{\,0}$ does not have the wDP$^*$ property.
\par (3) It is known that if the norm dual $E^{*}$ of a Banach lattice $E$ has the (weak) Dunford-Pettis property, so does $E$ (cf. \cite[Theorem 19.5]{AB} and \cite[Proposition 2]{W4}). This does not necessarily hold for the wDP$^*$ property. For example, $\ell^{1}=(c_{0})^*$ has the Schur property, but $c_{0}$ does not have the wDP$^{*}$ property.
\end{remark}

In \cite{CGL} it was proved that a Banach space $X$ has the DP$^*$ property if and only if every operator from $X$ into $c_0$ is a Dunford-Pettis operator. On the other hand, Wnuk \cite{W4} characterized the positive Schur property of a Banach lattice: a Banach lattice $E$ has the positive Schur property if, and only if, $E$ has order continuous norm and each continuous operator  $T:E\rightarrow c_0$ is almost Dunford-Pettis. Comparing this with Proposition \ref{Proposition 1} in the present paper we naturally posed the following theorem.

\begin{theorem}\label{Theorem 2.2}
For a $\sigma$-Dedekind complete Banach lattice $E$, the following statements are equivalent.

\begin{enumerate}
  \item $E$ has the wDP$^{\,*}$ property.
  \item Every continuous operator $T:E\rightarrow c_{\,0}$ is an almost Dunford-Pettis operator.
  \item Every positive operator $T:E\rightarrow c_{\,0}$ is an almost Dunford-Pettis operator.
  \item Every positive operator $T:E\rightarrow c_{\,0}$ is a Dunford-Pettis operator.
\end{enumerate}
\end{theorem}

\begin{proof}
$(2)\Rightarrow(3)$ Obvious.
\par$(3)\Leftrightarrow(4)$  This is because a positive operator $T$ from a Banach lattice into a discrete Banach lattice with order continuous norm (e.g., $c_0,\,\ell^{p}\quad 1\leq p<\infty$) is almost Dunford-Pettis if and only if $T$ is Dunford-Pettis. See Example 4 of almost Dunford-Pettis operators on p.\,230 of \cite{W4}.
\par$(3)\Rightarrow(1)$  To prove that $E$ has the wDP$^{\,*}$ property, in view of Theorem \ref{Theorem 3} it is enough to show that for each disjoint weakly null sequence $(x_n)\subset E^{+}$ and each disjoint $w^{*}$-null sequence $(f_n)\subset (E^{\,*})^+$, we have $f_{n}(x_{n})\rightarrow0$. To this end, let $T:E\rightarrow c_0$ be defined by $T(x)=(f_{n}(x))$ for any $x\in E$. Clearly, the thus defined operator $T$ is a positive operator. By hypothesis, $T$ is an almost Dunford-Pettis operator. Therefore, $\|T(x_n)\|\rightarrow0$ as $n\rightarrow\infty$, and hence $f_{n}(x_{n})\rightarrow0$, as desired.
\par$(1)\Rightarrow(2)$ Let $T:E\rightarrow c_0$  be an arbitrary continuous linear operator. To finish the proof, we have to show that $\|T(x_n)\|\rightarrow0$ for every disjoint weakly null sequence $(x_n)$ of $E$. Assume by way of contradiction that $(\|T(x_n)\|)$ does not tend to 0 for some disjoint weakly null sequence $(x_n)$ of $E$. Then, by passing to a subsequence if necessary, we can suppose that there would exist some $\varepsilon>0$ such that $\|T(x_n)\|>\varepsilon$ for all $n\in\mathbb{ N}$.  For every $n\in\mathbb{N}$, there exists a canonical projection, say $\pi_{k_n}$, from $c_0$ into its coordinate space $\mathbb{R}$ such that $\|T(x_n)\|=|\pi_{k_n}(T(x_{n}))|$. Applying the idea used in the proof of Proposition 2.1 in \cite{CGL}, we can show that the sequence $(k_n)\subset\mathbb{N}$ can not be bounded. Again by passing to a subsequence if necessary, we can suppose that $(k_n)$ is strictly increasing. Then $(\pi_{k_n}\circ T)$ is a $w^{*}$-null sequence of $E^{\,*}$. Note that $(x_n)$ is a disjoint weakly null sequence of $E$. So, in view of \cite[\,Ex.\,22, p.73]{AB} there exists a disjoint sequence $(f_n)$ in $E^{\,*}$  such that $$|f_n|\leq|\pi_{k_n}\circ T|,\quad f_{n}(x_n)=(\pi_{k_n}\circ T)(x_n)=\pi_{k_n}(T(x_n)).$$Since $\pi_{k_n}\circ T\xrightarrow{w^*}0$, by Lemma \ref{Lemma 1} we have $f_{n}\xrightarrow{w^*}0$ in $E^{\,*}$. By hypothesis that $E$ has the wDP$^{*}$ property, it follows
from Theorem \ref{Theorem 3} that $f_n(x_n)\rightarrow0$ as $n\rightarrow\infty$. On the other hand, $$|\,f_n(x_n)|=|(\pi_{k_n}\circ T)(x_n)|=|\pi_{k_n}(T(x_n))|=\|T(x_n)\|>\varepsilon>0.$$This leads to a contradiction. Hence,  $\|T(x_n)\|\rightarrow0$ for every disjoint weakly null sequence $(x_n)$ of $E$, that is, $T$ is an almost Dunford-Pettis operator.
\end{proof}

\par Let $\mathcal{L}(E,\,F)$ denote the Banach space of all continuous linear operators between Banach lattices $E$ and $F$, and let $\mathcal{L}^{r}(E,\,F)$ denote the linear subspace of all regular operators, i.e., operators which can be written as the differences of two positive operators. It may be asked whether $\mathcal{L}(E,\,F)=\mathcal{L}^{r}(E,\,F)$  holds in Theorem \ref{Theorem 2.2}. An earlier  result due to Wnuk \cite{Wn} states that, for a $\sigma$-Dedekind complete Banach lattice $E$, $\mathcal{L}(E,\,c_0)=\mathcal{L}^{r}(E,\,c_0)$ if and only if $E$ is a discrete Banach lattice with order continuous norm. Therefore, even though $\ell^{\,1},\,\ell^{\infty}$ and $L^{1}[0,\,1]$ are the right spaces for Theorem \ref{Theorem 2.2}, we can see that
$\mathcal{L}(\ell^1,\,c_0)=\mathcal{L}^{r}(\ell^1,\,c_0)$,\, $\mathcal{L}(\ell^{\infty},\,c_0)\neq\mathcal{L}^{r}(\ell^{\infty},\,c_0)$, \, $\mathcal{L}(L^{1}[0,\,1],\,c_0)\neq\mathcal{L}^{r}(L^{1}[0,\,1],\,c_0)$.
\par It should also be noted that the wDP property is inherited by a closed ideal of a Banach lattice $E$ (\cite[Proposition 3]{W4}), whereas the wDP$^*$  is not. Consider $c_{0}$ as a closed ideal of $\ell^{\infty}$. However, by Theorem \ref{Theorem 2.2}  it is not surprising that $\sigma$-Dedekind complete complemented  sublattices of a $\sigma$-Dedekind complete Banach lattice with the wDP$^*$ property have this property too.
\section*{Acknowledgement}
The first author would like to thank deeply Professor W. Wnuk at A. Mickiewicz University of Poland for his help during the preparation of this paper.
\vskip 7mm

\end{document}